\documentclass[11pt]{amsart}
\usepackage[margin=1in]{geometry}
\usepackage{amsfonts, amstext, amsmath, amssymb}
\usepackage{graphicx}

\newtheorem{lemma}{Lemma}
\newtheorem{theorem}[lemma]{Theorem}

\begin{document}

\title{Weak K\"{o}nig's lemma implies the uniform continuity theorem: a direct proof}
\author[Hendtlass]{Matthew Hendtlass}
 \address{School of Mathematics and Statistics,
    University of Canterbury,
    Christchurch 8041,
    New Zealand}
\email{matthew.hendtlass@canterbury.ac.nz}

\begin{abstract}
\noindent
We show in Bishop's constructive mathematics---in particular, using countable choice---that weak K\"{o}nig's lemma implies the uniform continuity theorem.
\end{abstract}

\maketitle

\bigskip
\noindent
In \cite{HD} Hannes Diener proved, as part of the programme of reverse constructive mathematics, that weak K\"{o}nig's lemma
\begin{quote}
 \textbf{WKL}: Every infinite, decidable, binary tree has an infinite path. 
\end{quote} 
implies the uniform continuity theorem
\begin{quote}
 \textbf{UCT}: Every pointwise continuous function $f:[0, 1]\rightarrow\mathbf{R}$ is uniformly continuous. 
\end{quote} 
in Bishop's constructive mathematics. Diener's proof relies on several other results in reverse constructive mathematics; we give a short, direct proof.

\begin{theorem}
\label{T}
Weak K\"{o}nig's lemma implies the uniform continuity theorem.
\end{theorem}

\noindent
The idea of our proof is simple: given a continuous function $f:[0,1]\rightarrow\mathbf{R}$ and some $\varepsilon>0$, we use WKL to focus in on a point where the function exhibits (almost) greatest variation within $\varepsilon$ and then use continuity at that point to find our modulus $\delta$ of uniform continuity for $\varepsilon$. Before proving Theorem \ref{T} we must set up some notation.

\bigskip
\noindent
For each $n>0$ we let $S_n=\{0,2^{-n},\ldots,1-2^{-n},1\}$ and we write $\mathcal{D}$ for the set $\cup\{S_n:n\in\omega\}$ of dyadic rationals. We define a one-one function $g$ from the set $2^{<\omega}$ of finite binary sequences to $\mathcal{D}$ by
 $$
  g(a)=\sum_{i=0}^{|a|-1} a(i)2^{-(i+1)}
 $$
\noindent
and associate $x\in S_n$ with the unique finite binary string $a$ of length $n$ such that $g(a)=x$. The \emph{sum} of two finite binary sequences $a,b$ each of length $n$ is the binary sequence $a\oplus b$ of length $2n$ given by
 $$
 a\oplus b(n)=\left\{\begin{array}{ll}a(n/2) & n\mbox{ is even}\\ b((n-1)/2) & n\mbox{ is odd}.\end{array}\right. 
 $$
Both $g$ and $\oplus$ extend to functions on infinite binary sequences and we make no notational distinction between the functions on finite and infinite sequences. We let $\pi_0,\pi_1$ be the left and right inverses of $\oplus$ respectively; that is $\pi_0(a\oplus b)=a$ and $\pi_1(a\oplus b)=b$ for all suitable pairs $a,b\in2^{<\omega}\cup2^\omega$. For $\alpha\in2^\omega$ and $n\in\mathbf{N}$, $\bar{\alpha}(n)$ denotes the unique binary sequence of length $n$ that $\alpha$ extends. The downward closure of subset $A$ of $2^{<\omega}$ is $A\downarrow = \{a\in2^{<\omega}: \exists_{a^\prime\in A}(a < a^\prime)\}$, where $a < a^\prime$ if $a^\prime$ extends $a$. 

\bigskip
\noindent
For a given function $f:[0, 1]\rightarrow\mathbf{R}$ we define a predicate $\varphi_f$ on $\mathbf{R}^+\times\mathbf{R}^+$ by
 $$
  \varphi_f(\varepsilon,\delta)\equiv \exists_{x,y\in[a,b]}(|x-y|<\delta\wedge|f(x)-f(y)|>\varepsilon);
 $$
$\varphi_f(\varepsilon,\delta)$ holds if we have a witness that $\delta$ is not a modulus of uniform continuity for $f,\varepsilon$. Thus to show that $f:[0, 1]\rightarrow\mathbf{R}$ is uniformly continuous we must, given any $\varepsilon>0$, find some $\delta>0$ such that $\varphi_f(\varepsilon,\delta)$ is false. The next lemma shows how we can use \textbf{WKL} to reduce the truth of $\varphi_f(\varepsilon,\delta)$ to whether or not some specific $x,y\in[0, 1]$ are witnesses of  $\varphi_f(\varepsilon,\delta)$.

\begin{lemma}\label{L1}
 \textbf{WKL} $\vdash$ Let $f:[0, 1]\rightarrow\mathbf{R}$ be pointwise continuous and let $\delta,\varepsilon$ be positive real numbers. Then there exist $x,y\in[a,b]$ such that if $\delta$ is not a modulus of uniform continuity for $f$, then $x,y$ witness this:
  $$
   \varphi_f(\varepsilon,\delta) \rightarrow |x-y|<\delta \wedge |f(x)-f(y)|>\varepsilon.
  $$
\end{lemma}

\begin{proof}
Using countable choice, construct a function $\gamma:\mathcal{D}^2\times\mathbf{N}\rightarrow2$ such that
 \begin{eqnarray*}
  \gamma(x,y,n)=0 & \Rightarrow & |x-y|>\delta-2^{-n} \vee |f(x)-f(y)|<\varepsilon+2^{-n},\\
  \gamma(x,y,n)=1 & \Rightarrow & |x-y|<\delta \wedge |f(x)-f(y)|>\varepsilon;
 \end{eqnarray*}
\noindent
further we may assume that $\gamma$ is non-decreasing in the third argument. So if $\gamma(x,y,n)=1$, then $x, y$ are witnesses of $\varphi_f(\varepsilon,\delta)$.
Using $\gamma$, we can construct an increasing binary sequence $(\lambda_n)_{n\in\mathbf{N}}$ such that 
 \begin{eqnarray*}
  \lambda_n=0 & \Rightarrow & \forall_{x,y\in S_n}\gamma(x,y,n)=0,\\
  \lambda_n=1 & \Rightarrow & \exists_{x,y\in S_n}\gamma(x,y,m)=1.
 \end{eqnarray*}
\noindent
Finally we construct a decidable binary tree $T$ as follows. If $\lambda_n=0$ we let $T_n=2^n$, and if $\lambda_{n-1}=1$ we set $T_n=\{\sigma*0:\sigma\in T_{n-1}\mbox{ and }|\sigma|=\mathrm{ht}(T)\}$. If $\lambda_n=1-\lambda_{n-1}$, we let $x,y$ be the minimal elements of $S_n$ such that $\gamma(x,y,n)=1$  and we set $T_n=(2^n\cup\{x\oplus y\})\downarrow$---the branch $x\oplus y$ is the unique branch of $T_n$ with length $\mathrm{ht}(T_n)$, and it codes the witnesses $x,y$ that $\delta$ is not a modulus of uniform continuity for $\varepsilon$. Then
 $$
 T=\bigcup_{n\in\mathbf{N}} T_n
 $$
\noindent
is an infinite decidable tree.

\bigskip
\noindent
Using \textbf{WKL} we can construct an infinite path $\alpha$ through $T$. Set $x=g(\pi_0\alpha),y=g(\pi_1\alpha)$. Suppose there exist $u,v\in[0,1]$ such that $|u-v|<\delta$ and $|f(u)-f(v)|>\varepsilon$. Since $f$ is pointwise continuous, there must exist such $u,v\in \mathcal{D}$. Hence $\gamma(u,v,n)=1$ for some $n\in\mathbf{N}$ such that $u,v\in S_n$, so $\lambda_n=1$. It now follows from the construction of $T$ that $x,y$ have the desired property.
\end{proof}

\bigskip
\noindent
We recall a result of Hajime Ishihara \cite{Ish}: \textbf{WKL} is equivalent to the longest path principle
 \begin{quote}
  \textbf{LPP}: Let $T$ be a decidable tree. Then there exists $\alpha\in2^\omega$ such that for all $n$, if $\bar{\alpha}(n)\notin T$, then $T\subset2^{<n}$.
 \end{quote}
To get a longest path for a decidable tree $T$ apply \textbf{WKL} to the decidable  tree
 $$
  \{a\in2^{<\omega}: a\in T \mbox{ or } \exists_{b \in T}(\vert b\vert = \mathrm{ht}(T) \wedge a > b)\},
 $$
where $\vert b\vert$ is the length of $b$ and $\mathrm{ht}(T)$ is the height of $T$.

\bigskip
\noindent
Here then is our \textbf{proof of Theorem \ref{T}}:

\bigskip
\begin{proof}
Let $f:[0,1]\rightarrow\mathbf{R}$ be a pointwise continuous function and fix $\varepsilon>0$. We define a function $J$ taking finite binary sequences to subintervals of $[0,1]$ inductively: $J_{()}=[0,1]$ and if $J_u=[p,q]$, then $J_{u*0}=[p,(p+q)/2]$ and $J_{u*1}=[(p+q)/2,q]$. By repeated application of the lemma, let $x_u,y_u\in J_u$ be such that
  $$
   \varphi_{f|_{J_u}}(\varepsilon,2^{-|u|}) \rightarrow |x_u-y_u|<2^{-|u|} \wedge |f(x_u)-f(y_u)|>\varepsilon / 2,
  $$
\noindent
and using countable choice construct a decidable tree $T$ such that
 \begin{eqnarray*}
  u\in T & \Rightarrow & 
  |f(x_u)-f(y_u)|>\varepsilon / 2-2^{-|u|}\\
  u\notin T & \Rightarrow & 
  |f(x_u)-f(y_u)|<\varepsilon / 2.
 \end{eqnarray*}
\noindent
Let $\alpha$ be a longest path of $T$, and let $\xi$ be the unique element of
 $$
  \bigcap_{n\in\mathbf{N}}J_{\overline{\alpha}(n)}.
 $$  
\noindent
Using the continuity of $f$ at $\xi$ we can find $\delta>0$ such that $|f(x)-f(y)|<\varepsilon/2$ for all $x,y\in(\xi-\delta,\xi+\delta)$; let $n$ be such that $2^{-n+1}<\max\{\delta,\varepsilon\}$. If $u=\overline{\alpha}(n)\in T$, then $|x_u-y_u|<\delta$ and $|f(x_u)-f(y_u)|>\varepsilon-2^{-n}>\varepsilon/2$ contradicting our choice of $\delta$. Hence $\overline{\alpha}(n)\notin T$, so $T\subset 2^{<n}$. It follows from Lemma \ref{L1} and the construction of $T$ that for all $x,y\in[0,1]$, if $|x-y|<2^{-n}$, then $|f(x)-f(y)|<\varepsilon$.
\end{proof}


\begin{thebibliography}{99}

\bibitem{HD}H. Diener, Weak K\"{o}nig's lemma implies the uniform continuity theorem, Computability 2(1), p. 9--13, 2013.

\bibitem{Ish}H. Ishihara,  `Weak K\"{o}nig's Lemma Implies Brouwer's Fan Theorem: A Direct Proof', Notre Dame Journal of Formal Logic 47(2), p. 249--252, 2006.

\end{thebibliography}
\end{document}